\documentclass[a4paper, 12pt]{amsart}

\usepackage{graphicx} 
\usepackage{amsmath}
\usepackage{amsthm}
\usepackage{amsfonts}
\usepackage{amssymb}
\usepackage{amscd}
\usepackage{color}
\usepackage[linktocpage]{hyperref}

\usepackage{comment}

\usepackage[T1]{fontenc}
\usepackage{libertine}

\usepackage{tikz}

\usepackage{here}

\newtheorem{thm}{Theorem}[section]
\newtheorem{cor}[thm]{Corollary}
\newtheorem{lem}[thm]{Lemma}
\newtheorem{prop}[thm]{Proposition}

\theoremstyle{definition}
\newtheorem{defn}[thm]{Definition}

\newtheorem{rem}[thm]{Remark}
\newtheorem{conj}[thm]{Conjecture}

\numberwithin{equation}{section}

\def\FF{\mathbb F}
\def\NN{\mathbb N}
\def\ZZ{\mathbb Z}
\def\QQ{\mathbb Q}
\def\RR{\mathbb R}
\def\CC{\mathbb C}
\def\cA{\mathcal{A}}
\def\cB{\mathcal{B}}
\def\cC{\mathcal{C}}
\def\cS{\mathcal{S}}
\def\cR{\mathcal{R}}
\def\cU{\mathcal{U}}
\def\cV{\mathcal{V}}

\def\flS{\mathcal{S}^\flat}
\def\flR{\mathcal{R}^\flat}

\def\GL{\operatorname{GL}}
\def\SL{\operatorname{SL}}
\def\PSL{\operatorname{PSL}}
\def\Dic{\operatorname{Dic}}
\def\Ker{\operatorname{Ker}}
\def\Tr{\operatorname{Tr}}

\def\beine{\mathbf 1}
\def\bi{\mathbf i}
\def\bj{\mathbf j}
\def\bk{\mathbf k}

\def\<{\langle}
\def\>{\rangle}
\def\too{\longrightarrow}

\begin{document}

\title{Applications of the finite specializations of the $q$-deformed modular group}

\author[T. Byakuno]{Takuma Byakuno} 
\address{Department of Mathematics, Kansai University, Osaka, 564-8680, Japan.}
\email{k241344@kansai-u.ac.jp,tatatagobyakuno0822@gmail.com}

\author[X. Ren]{Xin Ren}
\address{Osaka Central Advanced Mathematical Institute (OCAMI), Osaka Metropolitan University, 3-3-138 Sugimoto, Sumiyoshi-ku Osaka 558-8585, Japan}
\email{xin-ren@omu.ac.jp, xinren1213@gmail.com}

\author[K. Yanagawa]{Kohji Yanagawa} 
\address{Department of Mathematics, 
Kansai University, Osaka, 564-8680, Japan.}
\email{yanagawa@kansai-u.ac.jp}


\begin{abstract}
 Recently, Morier-Genoud and Ovsienko introduced the $q$-deformed modular group.
 For the construction, they first gave a group $G_q \subset \operatorname{GL}(2, \mathbb{Z}[q^{\pm}])$ and then set 
$\operatorname{PSL}_q(2,\mathbb{Z}):=G_q/Z(G_q)$. We show that, for $\zeta \in \mathbb{C}^*$, $\operatorname{PSL}_q(2,\mathbb{Z})|_{q=\zeta}$ is finite, if and only if so is $G_q(\zeta):=G_q|_{q=\zeta} \subset \operatorname{GL}(2,\mathbb{C})$, if and only if $\zeta=\zeta_n$ for $n=2,3,4,5$, where $\zeta_n$ is a primitive $n$-th root of unity. Although this result is essentially available in the existing literature, we give a precise proof for further developments.  
For example, we show that the set of traces of elements in $G_q(\zeta)$  and the set of the special values at $\zeta$ of the normalized Jones polynomials of rational links are finite if and only if $\zeta =\zeta_n$ for $n=2,3,4,5$ and also for $n=6$.  We also study the quantized Pythagorean triples.  
\end{abstract}

\maketitle

\section{Introduction}
For  $n \in \NN$, the Euler $q$-integer $[n]_q=1+q+\cdots +q^{n-1}$ is a very classical subject of mathematics. Recently, 
Morier-Genoud and Ovsienko \cite{MO1} introduced the $q$-deformation $\displaystyle \left[\frac{r}{s} \right]_q$ of an irreducible fraction $\frac{r}{s} \in \QQ$ which can be written the quotient of two (Laurent) polynomials as follows:
$$
\left[\displaystyle\frac{r}{s}\right]_q= \frac{\cR_{\frac{r}s}(q)}{\cS_{\frac{r}s}(q)}.
$$
To define the $q$-deformed rational  $\displaystyle\left[\frac{r}{s}\right]_q$, they introduced $q$-deformations of continued fractions and the modular group. 
It is a classical result that  $\mathrm{SL}(2,\mathbb{Z})$ is generated by 
$$
R:=\begin{pmatrix}
1 & 1 \\
0 & 1 \\
\end{pmatrix} \quad \text{and} \quad
S:=\begin{pmatrix}
0 & -1 \\
1 & 0 \\
\end{pmatrix}, 
$$
and it (and its quotient $\mathrm{PSL}(2,\mathbb{Z}):=\SL(2,\ZZ)/\{ \pm E_2\}$) acts on $\displaystyle \mathbb{Q}\cup \left\{\left(\frac{1}{0}\right)\right\}$ by 
\[
\displaystyle
\begin{pmatrix}
a & b \\
c & d \\
\end{pmatrix}\left(\frac{r}{s}\right)=\frac{ar+bs}{cr+ds}.\]

They define
$R_q:=\begin{pmatrix}
q & 1 \\
0 & 1 \\
\end{pmatrix}$ and $S_q:=\begin{pmatrix}
0 & -q^{-1} \\
1 & 0 \\
\end{pmatrix}$ as the $q$-deformations of $R$ and $S$. 
Let $G_q$ be the subgroup of $\mathrm{GL}(2,\mathbb{Z}[q^{\pm1}])$ generated by $R_q$ and $S_q$.
Then $Z(G_q)= \{ \pm q^n E_2 \mid n \in \ZZ \}$, and  
\cite{LM} showed that  $\mathrm{PSL}_q(2,\mathbb{Z}) :=G_q/Z(G_q)$ is isomorphic to the modular group  $\mathrm{PSL}(2,\mathbb{Z})$.
The $q$-deformed rationals  $\displaystyle\left[\frac{r}{s}\right]_q$ can be defined via $\mathrm{PSL}_q(2,\mathbb{Z})$, in particular, 
the entries of a matrix $M_q \in \PSL_q(2,\ZZ)$ are 
$\cR_{\frac{r}s}(q)$ or $\cS_{\frac{r}s}(q)$ for some $\frac{r}s \in  \QQ$, up to multiplication by $\pm q^n$ for $n \in \ZZ$. See \S2 for detail.

The notion of $q$-deformed rational numbers is further extended to arbitrary real numbers \cite{MO2} by some number-theoretic properties of irrational numbers. These works are related to many directions including triangulated categories~\cite{BBL}, the Markov-Hurwitz approximation theory \cite{K, LMOV, XR}, 
Jones polynomials of rational knots \cite{KR, MO1, BBL, XR2, KMRWY} and combinatorics of posets \cite{OEK1, OEK2}. 

 In this paper, for $\zeta \in \CC^*$, we study the subgroup 
$$G_q(\zeta):=\{M_q|_{q=\zeta} \mid M_q \in G_q \}  \subset \mathrm{GL}(2,\mathbb{C}).$$
Let $\zeta_n \in \CC$ be a primitive $n$-th roots of unity. 
Although the following result is essentially implicit in previous works \cite{LLS, L24, FK} (see Remark~\ref{existing results} below), we include a detailed proof because it serves as the foundation of new results presented in the latter half of the paper. 



\begin{thm}[{cf. \cite{LLS, L24,FK}}]\label{finite} The following hold. 

(1) For $\zeta \in \CC^*$, $G_q(\zeta)$ is a finite group if and only if $\zeta =\zeta_n$ for $n=2,3,4,5$. 

(2) We have $G_q(\zeta_2)\cap \SL(2, \ZZ)=G_q(-1) \cap \SL(2, \ZZ)\cong C_6$, 
$$G_q(\zeta_3) \cap \SL(2, \CC) \cong 
G_q(\zeta_4) \cap \SL(2, \CC) \cong \SL(2,\FF_3)$$
and 
$$G_q(\zeta_5) \cap \SL(2, \CC)  \cong \SL(2,\FF_5).$$
Here $\SL(2,\FF_3)$ (resp. $\SL(2,\FF_5)$) is called the ``binary tetrahedral group" (resp. ``binary icosahedral group"), and has order $24$ (resp. $120$).  Moreover, $G_q(\zeta_2) \cong D_6$, 
$G_q(\zeta_3) \cong \SL(2,\FF_3) \times C_3$,   $G_q(\zeta_4) \cong \SL(2, \FF_3) \rtimes C_4$, 
and $G_q(\zeta_5) \cong \SL(2,\FF_5) \times C_5$.  
\end{thm}

In \cite{MO1,KMRWY}, it has been shown that $\cS_{\frac{r}s}(\zeta_n)=0$ if and only if $s \, (=\cS_{\frac{r}s}(1)) \in n \ZZ$ for $n=2,3,4$. In the present paper, we show the following. 

\begin{cor}[Corollaries~\ref{zeta5} and \ref{trace for 3,4,5}] 
The following hold. 
\begin{itemize}
\item[(1)] $\cS_{\frac{r}s}(\zeta_5)=0$ if and only if  $s \in 5\ZZ$ and $r \equiv \pm 1 \pmod{5}$. 
\item[(2)] Let $M_q \in G_q$, and set $f(q):=\Tr M_q$. 
For $n=3,4,5$, $f(1) \in n\ZZ$ implies $f(\zeta_n)=0$. 
\end{itemize}
\end{cor}

The following result states that the case $n=6$ still exhibits a certain degree of finiteness.

\begin{thm}[Corollaries~\ref{trace finite} and \ref{flat finite}] Let $\frac{r}s>1$ be an irreducible fraction, and let $J_{\frac{r}s}(q)$ denote the normalized Jones polynomial of the rational link associated with $\frac{r}s$. 
For $\zeta \in \CC^*$, the following are equivalent. 
\begin{itemize}
\item[(1)] $\{ \Tr M \mid M \in G_q(\zeta)\}$ is finite, 
\item[(2)] $\{J_{\frac{r}s}(\zeta) \mid \frac{r}s \in \QQ\}$ is finite, 
\item[(3)] $\zeta=\zeta_n$ for $n=2,3,4,5,6$. 
\end{itemize}
\end{thm}

The above results naturally lead to the following conjecture.

\begin{conj}\label{[n]_q divides}
For $n \ge 2$, if $\cS_{\frac{r}s}(\zeta_n)=0$, then $[n]_q$  
divides $\cS_{\frac{r}s}(q)$ 
(equivalently, $\cS_{\frac{r}s}(\zeta_n^{\, k})=0$ for all $0 < k <n$). 
In particular, $\cS_{\frac{r}s}(\zeta_n)=0$ implies $s \in n \ZZ$.
\end{conj}

The conjecture is obvious when $n$ is prime, and we will show that it also holds for the composite numbers $n=4,6,8,9,10,15$ and $25$ (Proposition~\ref{evidence}).

For an irreducible fraction $\frac{r}s>1$, the integers $a:=2rs$, $b:=r^2-s^2$ and $c:=r^2+s^2$ form a Pythagorean triple, that is, $a^2+b^2=c^2$. Quantizing the triple $(a,b,c)$, Mathevet, Morier-Genoud and Ovsienko \cite{MMO} defined  $\cA_{\frac{r}s}(q), \cB_{\frac{r}s}(q), \cC_{\frac{r}s}(q) \in \ZZ[q]$. Theorem~\ref{ABC} states that $\cA_{\frac{r}s}(q)$ and $\cB_{\frac{r}s}(q)$ satisfy ``the 2,3,4,5 property" unconditionally, that is,   
\begin{center}
for $n=2,3,4,5$, \quad $\cA_{\frac{r}s}(1) \in n\ZZ$  $\Longleftrightarrow$ $\cA_{\frac{r}s}(\zeta_n)=0$, 
\end{center}
and the same is true for $\cB_{\frac{r}s}(q)$. 
 An analogue of Conjecture~\ref{[n]_q divides} is expected to hold for $\cA_{\frac{r}s}(q)$ and $\cB_{\frac{r}s}(q)$. 


\section{Preliminaries}\label{Pre}
At the beginning, we briefly recall the definition and basic properties of the $q$-deformed rational numbers. For details, see \cite{LM,MO1,MOV}.

An irreducible fraction $\frac{r}{s}>1$ with $r,s \in \NN$ has the negative continued fraction (or the {\it Hirzebruch-Jung continued fraction}) expansions as follows\\
$$
\frac{r}{s}=c_1-\frac{1}{c_2-\cfrac{1}{\ddots-\cfrac{1}{c_k}}}$$
with $c_i\geq 2$ for all $i$. This expansion is denoted by $[[c_1,c_2,\ldots,c_k]]$. 
For the matrices $R=\begin{pmatrix}
1 & 1 \\
0 & 1 \\
\end{pmatrix}$ and 
$S=\begin{pmatrix}
0 & -1 \\
1 & 0 \\
\end{pmatrix}$ given in \S 1,  
a rational number $\frac{r}{s}=[[c_1,\ldots,c_k]]>1$ can be expressed by the following formula:
\begin{equation}\label{R_r_over_s-1}
\frac{r}{s}=M(c_1,\ldots, c_k)\left(\frac{1}{0}\right),  
\quad \text{where $M(c_1,\ldots, c_k):=R^{\, c_1}SR^{\, c_2}S\cdots R^{\, c_{k}}S$.}
\end{equation}

Let $q$ be a formal parameter, for the matrices 
\[\displaystyle
R_q=\begin{pmatrix}
q & 1 \\
0 & 1 \\
\end{pmatrix}, \ \ \ \ 
S_q=\begin{pmatrix}
0 & -q^{-1} \\
1 & 0 \\
\end{pmatrix}, \ \ \ \
\]
we consider the subgroup  
$$G_q :=\<R_q, S_q\> \subset \mathrm{GL}(2,\mathbb{Z}[q^{\pm1}]).$$ Then we have $Z(G_q)=\{\pm q^n E_2 \mid n \in \ZZ\}$, and set 
$$\PSL_q(2,\ZZ):= G_q/Z(G_q).$$

\begin{prop}[Leclere \&  Morier-Genoud {\cite[Proposition 1.1]{LM}}]\label{PSLq} 
With the above notation, the assignment 
$R_q \longmapsto R$, $S_q \longmapsto S$, 
induces an isomorphism $$\mathrm{PSL}_q(2,\mathbb{Z}) \cong \mathrm{PSL}(2,\mathbb{Z}).$$
\end{prop}


For a sequence of integers $c_1, \ldots, c_k$, we set 
$$M_q(c_1,\ldots,c_k):=R_q^{\, c_1}S_q R_q^{\, c_2}S_q\cdots R_q^{\, c_k} S_q \in G_q \subset  \mathrm{GL}(2,\mathbb{Z}[q^{\pm 1}]).$$ 

\begin{defn}[\cite{MO1,MOV}]\label{def of q-deform}
{\normalfont
For an irreducible fraction $\frac{r}{s}=[[c_1,\ldots,c_k]]>1$, 
then the $q$-deformed rational number $\left[\displaystyle\frac{r}{s}\right]_q$ is defined by
$$
\left[\displaystyle\frac{r}{s}\right]_q= \frac{\cR_{\frac{r}s}(q)}{\cS_{\frac{r}s}(q)}, \quad 
\text{where} \quad M_q(c_1,\ldots, c_k)=\begin{pmatrix}
  \cR_{\frac{r}s}(q) & *\\
    \cS_{\frac{r}s}(q) & *
\end{pmatrix}.
$$
Hence, as a $q$-deformation of \eqref{R_r_over_s-1}, we have 
\[\displaystyle
\left[\displaystyle\frac{r}{s}\right]_q=M_q(c_1,\ldots, c_k)\left(\frac{1}{0}\right).
\]
}
\end{defn}

 We have $\mathcal{R}_{\frac{r}{s}}(q),\  \mathcal{S}_{\frac{r}{s}}(q)  \in \ZZ_{\ge 0}[q]$ with  $\mathcal{R}_{\frac{r}{s}}(1)=r$, $\mathcal{S}_{\frac{r}{s}}(1)=s$ and  $\mathcal{R}_{\frac{r}{s}}(0)=\mathcal{S}_{\frac{r}{s}}(0)=1$. 
These polynomials, which are in fact monics, can be seen as the rank polynomials of {\it fence posets} (cf. \cite{OEK2}).  
Therefore, results on $\mathcal{S}_{\frac{r}{s}}(q)$ inherit a combinatorial flavor.

The definition of the $q$-deformed rational number $\displaystyle\left[\frac{r}{s}\right]_q$ can be extended to the case where $\frac{r}{s} \leq 1$ including negative rational numbers by the following formulas (we always assume that $s \ge 1$):
\begin{equation}\label{TRANS}
\left[\frac{r}{s} +1\right]_q=q\left[\frac{r}{s}\right]_q+1.
\end{equation}
In particular, we set
$\displaystyle
\left[\frac{0}{1}\right]_q:=\frac{0}{1}$ and  
$\displaystyle
\left[\frac{1}{0}\right]_q:=\frac{1}{0}$.

In order to uniquely determine $\cR_{\frac{r}{s}}(q)$ and $\mathcal{S}_{\frac{r}{s}}(q)$ from 
 $$
\left[\frac{r}{s}\right]_q=\frac{\mathcal{R}_{\frac{r}{s}}(q)}{\mathcal{S}_{\frac{r}{s}}(q)}, 
$$
we impose the following conditions:
\begin{itemize}
\item $\mathcal{R}_{\frac{r}{s}}(q)$ and $\mathcal{S}_{\frac{r}{s}}(q)$ are coprime (Laurent) polynomials with integer coefficients.  
\item $\mathcal{S}_{\frac{r}{s}}(q) \in \ZZ[q]$ with $\mathcal{S}_{\frac{r}{s}}(0)=1$. 
\end{itemize}

Note that $\mathcal{S}_{\frac{r}{s}+n}(q)=\mathcal{S}_{\frac{r}{s}}(q)$ for all $n \in \ZZ$. 
For $\frac{r}s<0$, $\cR_{\frac{r}s}(q) \in \ZZ_{\le 0}[q^{-1}]$. Similarly, for $0 < \frac{r}s<1$, we have $\cR_{\frac{r}s}(q) \in \ZZ_{\ge 0}[q]$ with $\cR_{\frac{r}s}(0)=0$. 
For a general $\frac{r}s$, we have 
\begin{equation}\label{r and s}
\cR_{\frac{r}s}(q)=\pm q^n \cS_{-\frac{s}r}(q) \qquad (\exists n \in \ZZ),   
\end{equation}
and hence 
\begin{equation}\label{R and S}
\left \{ \cS_{\frac{r}s}(q) \, \middle | \,  \frac{r}s \in \QQ \right \}= 
\left \{ \cS_{\frac{r}s}(q) \, \middle | \,  \frac{r}s >1 \right \}=
\left\{ \cR_{\frac{r}s}(q) \, \middle | \,   \frac{r}s >1 \right \}.
\end{equation}

While the following fact is already implicit in earlier works of Morier-Genoud and Ovsienko, it was made explicit in  \cite{RY}.

\begin{prop}[{c.f. \cite[Corollary~3.3]{RY}}]\label{column of a general element} 
For 
$$\begin{pmatrix}
\cR(q) & \mathcal{V}(q) \\
\cS(q) &  \mathcal{U}(q) \\
\end{pmatrix} \in \mathrm{PSL}_{q}(2,\mathbb{Z})
$$   
with $r:=\cR(1)$, $s:=\cS(1)$, $v:=\mathcal{V}(1)$ and 
$u:=\mathcal{U}(1)$, we have 
$$
\left[\displaystyle\frac{r}{s}\right]_q= \frac{\cR(q)}{\cS(q)}  \quad \text{and} \quad  \left[\displaystyle\frac{v}{u}\right]_q= \frac{\mathcal{V}(q)}{\mathcal{U}(q)}. 
$$ 
\end{prop}

For a polynomial $f(q)\in\mathbb{Z}[q]$, we set \[ f(q)^{\vee}:=q^{\mathsf{deg}(f)}f(q^{-1}) \in \ZZ[q].\]  
If $f(0) \ne 0$, we have $(f(q)^\vee)^\vee=f(q)$. 
We say $f(q)$ is  {\it palindromic} if $f(q)=f(q)^{\vee}$. For $\frac{r}s>1$, we have
\begin{equation}\label{s/r full}
\cR_{\frac{s}r}(q)= q^d \cS_{\frac{r}s}(q^{-1}) \quad \text{and} \quad   \cS_{\frac{s}r}(q)= q^d \cR_{\frac{r}s}(q^{-1})  \quad   \text{for} \quad  d := \deg \cR_{\frac{r}s}(q), 
\end{equation}
in particular, 
\begin{equation}\label{s/r}
\cS_{\frac{s}r}(q)=\cR_{\frac{r}s}(q)^\vee.    
\end{equation}

\begin{thm}[Leclere \&  Morier-Genoud {\cite[Theorem~3]{LM}}]\label{Tr A}
For $M_q \in G_q$, the trace $\Tr M_q$ is a palindromic polynomial whose coefficients are non-negative up to a multiple of $\pm q^n$ for some $n \in \ZZ$. 
\end{thm}

The rank polynomial of a {\it circular fence poset} equals the trace  $\Tr M_q$ of some $M_q \in G_q$ (see \cite{OEK1}). 
In this paper, we will give several results on $\Tr M_q$, and  they have a combinatorial side via the above interpretation.

\medskip

 Bapat, Becker and Licata \cite{BBL} gave another $q$-deformation $$\left[\displaystyle\frac{r}{s}\right]^\flat_q=\frac{\flR_{\frac{r}{s}}(q)}{\flS_{\frac{r}{s}}(q)}$$ of $\frac{r}s$. As before, we have $\flS_{\frac{r}{s}}(q) \in \ZZ_{\ge 0}[q]$ with $\flS_{\frac{r}{s}}(0)=1$.  
 To emphasize the contrast with $\left[\displaystyle\frac{r}{s}\right]^\flat_q$, one should denote the original deformation $\left[\displaystyle\frac{r}{s}\right]_q$ by  $\left[\displaystyle\frac{r}{s}\right]_q^\sharp$; however, for simplicity, we just write $\left[\displaystyle\frac{r}{s}\right]_q$ throughout this paper. 
 The same is true for $\cR_{\frac{r}s}(q)$ and $\cS_{\frac{r}s}(q)$. 

A presentation of $\left[\dfrac{r}s \right]^\flat_q$ directly using $\left[\dfrac{r}s \right]_q$  was given  by Thomas \cite{T24}, subsequently refined by Jouteur \cite{J25}, and the second and third authors of the present paper. The following is found in ``another proof of Theorem 4.2'' of \cite{RY}. 
For $\frac{r}s=\alpha>1$, we have 
\begin{equation}\label{Thomas}
\begin{pmatrix}
\flR_\alpha(q)^\vee \\
q^{v(\alpha)} \flS_\alpha(q)^\vee
\end{pmatrix}
=
\begin{pmatrix}
q & 1-q \\
q-1 &1
\end{pmatrix}
\begin{pmatrix}
\cR_\alpha(q) \\
\cS_\alpha(q)
\end{pmatrix}, 
\end{equation}
where $v(\alpha)$ is the maximam number $n$ such that $x^n$ divides $(q-1)\cR_\alpha(q)+\cS_\alpha(q)$.  


 The equations corresponding to \eqref{r and s}, \eqref{R and S} and \eqref{s/r} also hold for $\flR_{\frac{r}s}(q)$ and $\flS_{\frac{r}s}(q)$ (c.f. \cite[Proposition 2.11.]{RY}). Especially, we have 
\begin{equation}\label{R and S for flat}
\left \{ \flS_{\frac{r}s}(q) \, \middle | \,  \frac{r}s \in \QQ \right \}=
\left \{ \flS_{\frac{r}s}(q) \, \middle | \,  \frac{r}s >1 \right \}=
\left\{ \flR_{\frac{r}s}(q) \, \middle | \,   \frac{r}s >1 \right \}.
\end{equation}

An irreducible fraction $\frac{r}s >1$ determines a {\it rational link} (or {\it 2-bridge link}) $L({\frac{r}s})$. For more details, see \cite[\S4]{KR} and the references cited therein.
As a useful isotopy invariant for an oriented link $L$ in $\mathbb{S}^3$, the Jones polynomial $V_L(t)~\in \ZZ[t^{\pm 1}]\cup t^{\frac12}\ZZ[t^{\pm 1}]$ is well-studied. 
Lee and Schiffler~\cite[Proposition 1.2 (b)]{KR} introduced the following normalization $J_{{\frac{r}s} }(q)$ of the Jones polynomial $V_{L({\frac{r}s} )}(t)$ of a rational link  $L({\frac{r}s} )$: 
\begin{equation}\label{def normalized Jones}
 J_{{\frac{r}s} }(q):=\pm t^{h}V_{L({\frac{r}s})}(t)|_{t=-q^{-1}},
\end{equation}
where $\pm t^{h}$ is the leading term of $V_{L({\frac{r}s})}(t)$, and the sign is chosen so that the constant term of $J_{\frac{r}s}(q)$ is 1.

\begin{thm}[Morier-Genoud \& Ovsienko {\cite[Proposition A.1]{MO1}}, Bapat, Becker \& Licata {\cite[Theorem~A3]{BBL}}] \label{J-L}
For a rational number ${\frac{r}s}>1$, we have 
\begin{equation}\label{Jones and flat}
J_{\frac{r}s}(q)=q \cR_{\frac{r}s}(q)+(1-q)\cS_{\frac{r}s}(q)=\flR_{\frac{r}s} (q)^\vee=\flS_{\frac{s}r}(q).
\end{equation}
\end{thm}

Here the last equality of \eqref{Jones and flat} follows from the $\flat$ version of \eqref{s/r} (c.f. \cite[Proposition 2.11.(3)]{RY}). Finally, we remark that our study in this paper is closely related to the following historical result. See the classical book \cite{Kl} of F. Klein.  

\begin{thm}[Classification of  finite subgroups of $\SL(2,\CC)$]\label{finite SL(2-C)}
The finite subgroups of $\SL(2,\CC)$, up to conjugacy, are precisely: the cyclic groups $C_n$, the binary dihedral groups (or dicyclic groups) $\Dic_n$ of order $4n$, the binary tetrahedral group (isomorphic to $\SL(2,\FF_3)$) of order $24$, the binary octahedral group of order $48$, and the binary icosahedral group (isomorphic to $\SL(2,\FF_5)$) of order $120.$
\end{thm}


\section{Finite specializations of the \texorpdfstring{$q$}{q}-deformed modular group}

This section is devoted to the proof of Theorem~\ref{finite}.

For $\zeta \in \CC^*$, $R_q|_{q=\zeta}$ has a finite order if and only if $\zeta =\zeta_n$ for some $n \ge 2$.  Hence, if $G_q(\zeta)$ is a finite group, then  $\zeta =\zeta_n$ for some $n \ge 2$. First, we consider the cases $n=2,3,4$, and 5. Each case is treated in a separate subsection. Finally, we consider the case $n \ge 6$ in the final subsection.

\subsection{The case \texorpdfstring{$n=2$}{n=2}} Clearly, $\zeta_2=-1$, and $G_q(-1)$ is generated by 
$$
R_{-1}:=R_q|_{q=-1}=\begin{pmatrix} -1 &1 \\ 0 &1 \end{pmatrix} \qquad \text{and} \qquad
S_{-1}:=S_q|_{q=-1}=\begin{pmatrix}0&1\\1&0 \end{pmatrix}. 
$$
An easy calculation shows that  $R_{-1}S_{-1}=\begin{pmatrix} 1 &-1 \\ 1 & 0 \end{pmatrix}$ has order 6, and generates $H:=G_q(-1) \cap \SL(2,\ZZ)$. Moreover, $\<R_{-1} \> \cdot H =G_q(-1)$ and $G_q(-1)=H \rtimes \<R_{-1} \>   \cong D_6$.   \qed

\subsection{The case \texorpdfstring{$n=3$}{n=3}}
In the rest of the paper, $\omega$ means a primitive cube root of unity $\zeta_3$. That is, $\omega =\dfrac{-1 \pm \sqrt{-3}}2$. 

The group $G_q(\omega)$ is the subgroup of $\mathrm{GL}(2, \CC)$ generated by 
$$
R_\omega:=R_q|_{q=\omega}=\begin{pmatrix}\omega &1 \\ 0 &1 \end{pmatrix} \qquad \text{and} \qquad
S_\omega:=S_q|_{q=\omega}=\begin{pmatrix}0&-\omega^2\\1&0 \end{pmatrix}. 
$$ 
A direct computation shows that $G_q(\omega)$ is a group of order 72 whose subgroup $H:=G_q(\omega) \cap \SL(2, \CC)$ consists of the following 24 elements  
$$
\pm \begin{pmatrix}1&0\\0&1 \end{pmatrix}, \quad  \pm \begin{pmatrix} \omega^2 & \omega \\ 0& \omega \end{pmatrix}, \quad 
\pm \begin{pmatrix}0&\omega \\ -\omega^2 &0 \end{pmatrix}, \quad 
\pm \begin{pmatrix}\omega& -\omega \\0& \omega^2 \end{pmatrix},
$$
$$ 
\pm \begin{pmatrix}1&-1\\1&0 \end{pmatrix}, \quad 
\pm \begin{pmatrix}0&-\omega^2 \\ \omega &1 \end{pmatrix}, \quad 
\pm \begin{pmatrix}1&\omega^2\\-\omega&0 \end{pmatrix}, \quad  
\pm \begin{pmatrix} \omega & 0 \\ -\omega^2 & \omega^2 \end{pmatrix},
$$

$$
\pm \begin{pmatrix} \omega^2 & 0 \\ \omega^2 & \omega \end{pmatrix}, \quad  \pm \begin{pmatrix}0&-1 \\ 1 & -1 \end{pmatrix}, \quad \pm \begin{pmatrix} -\omega^2 &\omega^2 \\ 1 & \omega^2 \end{pmatrix}, \quad 
\pm \begin{pmatrix} \omega& 1  \\ \omega & -\omega \end{pmatrix}.$$
To show that $H \cong \SL(2,\FF_3)$, one can use the classification theorem (Theorem~\ref{finite SL(2-C)}). For example, it is enough to compare the traces of the above  matrices with the character tables of the corresponding groups.   
However, we take an alternative approach here. 
From the above matrices, we set
$$\beine := E_2, \quad \bi:=  \begin{pmatrix}0&\omega \\ -\omega^2 &0 \end{pmatrix}, \quad \bj:=\begin{pmatrix} -\omega^2 &\omega^2 \\ 1 & \omega^2 \end{pmatrix}, \quad \bk:= \begin{pmatrix} \omega& 1  \\ \omega & -\omega \end{pmatrix} \in H.$$
Then we  have $$\bi^2=\bj^2=\bk^2=\bi\bj\bk=-\beine,$$
and hence $\{a\beine+b\bi+c\bj+d\bk \mid a,b,c,d \in \RR \}$ forms a quaternion. 
An easy calculation shows that $H \setminus \{ \pm \beine, \pm \bi, \pm \bj, \pm \bk\}$
consists of 16 elements given by 
$\frac12(\pm \beine \pm \bi \pm \bj \pm \bk).$ 
So we have $H \cong \mathrm{SL}(2,\FF_3)$ by \cite[\S 11.2]{V}. 
We also have $H \triangleleft G_q(\omega)$, $\omega E_2 \in Z(G_q(\omega))\setminus H$, and $G_q(\omega) = H \times \<\omega E_2\> \cong \SL(2,\FF_3) \times C_3.$


\subsection{The case \texorpdfstring{$n=4$}{n=4}} In the rest of the paper, $i$ means $\sqrt{-1}=\zeta_4$. 
A direct computation shows that $G_q(i)$ is a group of order 96 whose subgroup $H:=G_q(i) \cap \SL(2, \CC)$ consists of the following 24 elements
$$
\pm \begin{pmatrix}1&0\\0&1 \end{pmatrix}, \quad  
\pm \begin{pmatrix}i&1-i \\ 0 & -i \end{pmatrix}, \quad  
\pm \begin{pmatrix}i& 0 \\1+i & -i \end{pmatrix}, \quad  
\pm\begin{pmatrix} -1 & 1+i \\ -1+i& 1 \end{pmatrix}, 
$$
$$ 
\pm \begin{pmatrix} -1 &i \\ i & 0 \end{pmatrix}, \ 
\pm \begin{pmatrix} i& -i  \\ 1 & -1-i \end{pmatrix}, \   
\pm \begin{pmatrix}0&-1 \\ 1 &-1 \end{pmatrix}, \ 
\pm \begin{pmatrix} 1+i & -i \\ 1 & -i \end{pmatrix},
$$
$$
\pm  \begin{pmatrix}1&-1\\1&0 \end{pmatrix}, \  
\pm \begin{pmatrix}1-i&-1 \\ -i & i \end{pmatrix}, 
 \  
\pm \begin{pmatrix} 0 & i \\ i & 1 \end{pmatrix}, \  
\pm \begin{pmatrix}i&1\\i&1-i \end{pmatrix}.$$

Set 
\begin{equation}\label{ijk}
\beine := E_2, \quad \bi:= \begin{pmatrix}i&1-i \\ 0 & -i \end{pmatrix}, \ \bj:=\begin{pmatrix}i& 0 \\1+i & -i \end{pmatrix}, \ \bk:=\begin{pmatrix} -1 & 1+i \\ -1+i& 1 \end{pmatrix}  \in H.
\end{equation}
Then $\{a\beine+b\bi+c\bj+d\bk \mid a,b,c,d \in \RR \}$ forms a quaternion, and we can show that $H \cong   \mathrm{SL}(2,\FF_3)$ by a similar argument to the case $n=3$. 


We have  $\{ X \in G_q(i) \mid \det(X)=-1\}= iE_2 \cdot H$, 
$$\{ X \in G_q(i) \mid \det(X)=i\}=Y_1 \cdot H =\{ \pm  Y_k \mid 1 \le k \le 12 \}$$ and  
$\{ X \in G_q(i) \mid \det(X)
=-i\}=iY_1 \cdot H =\{ \pm i Y_k \mid 1 \le k \le 12 \},$ where 
$$Y_1:=R_q|_{q=i}= \begin{pmatrix}i &1\\ 0 & 1\end{pmatrix}, \ Y_2:=   \begin{pmatrix}1 &-1\\ 0 & i\end{pmatrix},  \ Y_3 :=  \begin{pmatrix}i &-i\\ i+1 & -i\end{pmatrix}, \  Y_4 := 
 \begin{pmatrix}-1 &i\\ i-1 & 1\end{pmatrix},$$  
 $$Y_5:= \begin{pmatrix}0 &1\\ -i & 0\end{pmatrix}, \ Y_6:= \begin{pmatrix}0 &i\\ -1 & i+1\end{pmatrix}, \ 
Y_7:= \begin{pmatrix}-1 &i+1\\ -1 & 1\end{pmatrix}, 
\  Y_8:=  \begin{pmatrix}i  &1-i\\ 1 & -i\end{pmatrix},$$
$$Y_9:= \begin{pmatrix}i+1 &-i\\ 1 & 0\end{pmatrix}, \  Y_{10}:= \begin{pmatrix}1 &0\\ -i & i\end{pmatrix}, \ Y_{11}:= \begin{pmatrix}i &0\\ i & 1\end{pmatrix}, \  Y_{12}:=  \begin{pmatrix}i-1 &1\\ i & 1-i\end{pmatrix}. \quad $$
 Note that  $H \triangleleft G_q(i)$, $\< Y_1 \> \cap H =\{ E_2\}$, $\< Y_1 \> \cdot H=G_q(i)$, $\<Y_1\> \cong C_4$. 
 So we have $G_q(i)= H \rtimes \< Y_1 \> \cong \SL(2, \FF_3) \rtimes C_4$. 
 Easy calculation shows that $Z(G_q(i))=\{\pm E_2, \pm iE_2\}$, 
 while $Z( \SL(2, \FF_3)\times C_4) = \{\pm E_2\} \times C_4$. 
Hence $G_q(i) \not \cong \SL(2, \FF_3) \times C_4$.
 \qed 

\begin{rem}
Set $H:=G_q(i) \cap \SL(2, \CC)$ as in the above proof, and set $J:=\{M \in G_q(i)  \mid \det M=\pm 1\}= H \sqcup i H$. Note that $J= H \times \< \, i \, \bi \, \> \cong \SL(2,\FF_3) \times C_2$, where $\bi \in H$ is the one given in \eqref{ijk}.  
We also have $H/\{ \pm E_2 \} \cong A_4$ and $G_q(i)/\{\pm E_2\} \cong A_4 \rtimes C_4$. Hence it seems that there is no direct relation between our $G_q(i)$ and the binary octahedral group, whose order is 48. 
Note that the binary octahedral group appears in the context of Burau's representation at $i$ (\cite[Proposition~3.1]{FK}). 
\end{rem}


\subsection{The case \texorpdfstring{$n=5$}{n=5}}
To prove this case, we use computer calculations, and the relevant files are available at:\\
\begin{minipage}{0.9\textwidth}
\footnotesize
\url{https://drive.google.com/drive/folders/1vaeBPDucbgmqz65iPUXneeKLmZ23JB1s?usp=sharing}
\end{minipage}
\*\\

We can check that
\[
R_{\zeta_5}:=R_q\big|_{q=\zeta_5}=
\begin{pmatrix}
\zeta_5 & 1\\
0 & 1
\end{pmatrix},
\qquad
S_{\zeta_5}:=S_q\big|_{q=\zeta_5}=
\begin{pmatrix}
0 & -\zeta_5^{\, -1}\\
1 & 0
\end{pmatrix}
\]
generate a finite subgroup of order $600$ in $\mathrm{GL}(2,\mathbb{C})$.
Thus, we have $|G_q(\zeta_5)|=600$, and $H:=G_q(\zeta_5)\cap \mathrm{SL}(2,\mathbb{C})$ has order $120$. 
Since $H$ is non-abelian, it is isomorphic to either the binary dihedral group $\Dic_{30}$ or the binary icosahedral group  $\SL(2, \FF_5)$ by the classification of finite subgroups of $\mathrm{SL}(2,\CC)$ (Theorem~\ref{finite SL(2-C)}).
The former possibility can be excluded. Indeed, we will show  that $H/\{ \pm E_2\} \cong \PSL(2, \ZZ/5\ZZ) \cong A_5$ in Remark~\ref{PSL(Z/nZ)} below, whereas no quotient group of $\Dic_{30}$ is isomorphic to $A_5$. 
We can also show that $H \cong \SL(2,\FF_5)$ by a more direct argument. We can find
$$
A:=
\begin{pmatrix}
0 & 1\\
-1 & 1
\end{pmatrix}
=S_{\zeta_5}^2 R_{\zeta_5}^2 S_{\zeta_5} R_{\zeta_5} S_{\zeta_5} R_{\zeta_5}^2 S_{\zeta_5},
\qquad
B:=
\begin{pmatrix}
-\zeta_5^2 & 0\\
\zeta_5^3 & -\zeta_5^3
\end{pmatrix}
=S_{\zeta_5} R_{\zeta_5} S_{\zeta_5}^5
$$
in $H$, which satisfy $A^3 = B^5 = (AB)^2 = -E_2$ and $H=\<A,B\>$. 
Therefore, $H \cong \SL(2, \FF_5)$. 
We can also verify that $G_q(\zeta_5)=H\times \<\, \zeta_5 E_2 \, \> \cong \mathrm{SL}(2,\FF_5) \times C_5$.

\subsection{The case \texorpdfstring{$n\ge 6$}{n>=6}}
It remains to show that $G_q(\zeta_n)$ is infinite for $n \ge 6$. It has been checked in \cite[Remark 8.5]{KMRWY} that $G_q(\zeta_6)$ is an infinite group. In fact, $M:=(R^{3}_q S_q)|_{q=\zeta_6} \in G_q(\zeta_6)$ is not diagonalizable, and hence $M^m \ne E_2$ for all positive integers $m$. So we may assume that $n \ge 7$. 
In the rest of the paper, for $\zeta \in \CC^*$,  $[j]_\zeta$ denotes $[j]_q|_{q=\zeta}$. 

\begin{lem}\label{lemma_j-1}
For any integer $n\geq7$, there is 
a positive integer $j$ with $0\leq j \leq n-1$ such that $\left\vert   \left[j \right]_{\zeta_{n}}   \right\vert>2$.  
\end{lem}

\begin{proof}
We note that
\begin{align*}
\displaystyle
\left\vert 1 - \zeta_{n}^{\, j}   \right\vert&=\left\vert 1 - \left( \mathrm{cos}\left(\frac{2j\pi }{n}\right)+i \mathrm{sin}\left(\frac{2j\pi }{n}\right) \right)   \right\vert \\
&=\sqrt{\left(1 -  \mathrm{cos}\left(\frac{2j\pi}{n}\right) \right)^2+ \mathrm{sin}^2\left(\frac{2j\pi}{n}\right)}\\
&=\sqrt{2-2\, \mathrm{cos}\left(\frac{2j\pi }{n}\right)}=2\, \mathrm{sin}\left(\frac{j\pi }{n}\right).
\end{align*}
So we have 
$$\displaystyle
\left\vert   \left[j \right]_{\zeta_{n}}   \right\vert
=\frac{\left\vert 1 - \zeta_{n}^{\,j}\right \vert}{\left\vert 1 - \zeta_{n}\right \vert}
=\frac{\mathrm{sin}\left(\frac{j \pi }{n}\right)}{\mathrm{sin}\left(\frac{\pi}{n}\right)}.
$$
If $n\geq 7$ and odd, we take $j=\frac{n+1}{2}$. Since $n \ge 7$ now, we have 
\begin{eqnarray*}
\left\vert   \left[j \right]_{\zeta_{n}}   \right\vert =
\frac{\sin\left(\frac{\pi (n+1)}{2n}\right)}{\sin\left(\frac{\pi}{n}\right)}
=\frac{\sin \left(\frac{\pi }{2}+\frac{\pi}{2n}\right)}{\sin\left(\frac{\pi}{n}\right)}=\frac{\cos\left(\frac{\pi}{2n}\right)}{\sin\left(\frac{\pi}{n}\right)}
&=&\frac{\cos\left(\frac{\pi}{2n}\right)}{2\, \sin\left(\frac{\pi}{2n}\right)\cos\left(\frac{\pi}{2n}\right)}\\
&=&\frac{\csc\left(\frac{\pi}{2n}\right)}{2} > 2. 
\end{eqnarray*}
If $n\geq 7$ and even,  we  take $j=\frac{n+2}{2}$. 
Then,  
$$
\left\vert   \left[j \right]_{\zeta_{n}}   \right\vert =
\frac{\mathrm{sin}\left(\frac{\pi (n+2)}{2n}\right)}{\mathrm{sin}\left(\frac{\pi}{n}\right)}=\frac{\mathrm{sin}\left(\frac{\pi }{2}+\frac{\pi}{n}\right)}{\mathrm{sin}\left(\frac{\pi}{n}\right)}=\frac{\mathrm{cos}\left(\frac{\pi}{n}\right)}{\mathrm{sin}\left(\frac{\pi}{n}\right)}=\mathrm{cot}\left(\frac{\pi}{n}\right)>2. 
$$ 
\end{proof}

\noindent{\bf The final step of the proof of Theorem~\ref{finite}:}
Assume that $n \ge 7$. For a positive integer $j$, we have 
$$X_j := (R^{\, j}_q S_q)|_{q=\zeta_n}=\left( \begin{array}{cccc}[j]_{\zeta_n}&-\zeta_n^{\, j-1}\\1&0 \end{array}\right)\in G_q(\zeta_n).$$
The eigenvalues $\lambda, \lambda'$ of $X_j$ are the roots of the equation 
\begin{equation}\label{CHEQ}
x^2-\left[ j \right]_{\zeta_{n}}x+\zeta^{\, j-1}_{n}=0. 
\end{equation}
By Lemma~\ref{lemma_j-1}, we have
$$
\left\vert\lambda +\lambda'\right\vert = 
\left\vert   \left[ j \right]_{\zeta_{n}}   \right\vert >2$$ 
for some $j$. Since  $\vert\lambda \lambda'\vert=\vert\zeta^{\, j-1}_{n}\vert=1$, we may assume that $\vert\lambda\vert >1$ and $\vert \lambda'\vert <1$. 
So we have $\vert\Tr (X_j^{\, m})\vert \to \infty$ ($m
\to \infty$), and hence 
$\{ \Tr M \mid M \in G_q(\zeta_n)\}$ is an infinite set, and $G_q(\zeta_n)$ is an infinite group. \qed

\medskip



\begin{cor}\label{S finite} 
For $\zeta \in \CC^*$, 
the set $\{ \cS_{\frac{r}s}(\zeta) \mid \frac{r}s >1 \} \ (=\{ \cR_{\frac{r}s}(\zeta) \mid \frac{r}s >1 \} )$ is finite if and only if $\zeta=\zeta_n$ for $n=2,3,4, 5$. 
\end{cor}

\begin{proof}
If $\zeta \ne \zeta_n$ for all $n \ge 2$, then $[l]_\zeta \ne 0$ for all $l$, and $\{ [m]_\zeta \mid m \in \NN\}$ is infinite. Since $\cS_{\frac{1}m}(q)=[m]_q$, if $\{ \cS_{\frac{r}s}(\zeta) \mid\frac{r}s >1 \}$ is finite, then  $\zeta =\zeta_n$ for some $n \ge 2$. 
So we may assume that $\zeta=\zeta_n$.
Since $G_q(\zeta_n)$ is finite if and only if so is $\{ \cS_{\frac{r}s}(\zeta_n) \mid \frac{r}s >1 \}$ by \eqref{r and s} and \eqref{R and S}, the assertion follows from Theorem~\ref{finite}. 
\end{proof}

\begin{cor}\label{barG-q}
For $\zeta \in \CC^*$,  $\overline{G}_q(\zeta):=\PSL_q(2,\ZZ)|_{q=\zeta}$ is a finite group if and only if $\zeta=\zeta_n$ for $n=2,3,4, 5$. 
\end{cor}

\begin{proof}
Considering the image of $R_q^{\, m}$ in $\overline{G}_q(\zeta)$, we see that $\overline{G}_q(\zeta)$ is infinite unless $\zeta=\zeta_n$ for some $n \ge 2$. 
So we may assume that $\zeta=\zeta_n$. Since 
$\overline{G}_q(\zeta_n)$ is finite if and only if so is $G_q(\zeta_n)$, the assertion follows from Theorem~\ref{finite}.  \end{proof}

If $n \ge 7$, $\{ \,  \Tr M \mid M \in G_q(\zeta_n) \, \}$ is unbounded as a subset of $\CC$ by the last step of the proof of Theorem~\ref{finite}, and hence $\{ \, \cS_{\frac{r}s}(\zeta_n) \mid \frac{r}s >1  \, \}$ is also. 
 Since $\ZZ[\zeta_6] \, (=\ZZ[\omega])$ is discrete in $\CC$ and  $\{ \, \cS_{\frac{r}s}(\zeta_6) \mid \frac{r}s>1 \, \}$ is infinite, the latter set cannot be bounded. Hence $\{ \, \cS_{\frac{r}s}(\zeta_n) \mid \frac{r}s >1 \,  \}$ is unbounded for $n \ge 6$.

\begin{rem}\label{existing results}
Leclere \cite{L24} showed a result corresponding to Theorem~\ref{finite}, but she directly treated $\PSL_q(2,\ZZ)|_{q=\zeta_n}$.  Topkara and Uludag \cite{TU25} also announced a similar observation. So $\SL(2, \FF_3)$ and $\SL(2, \FF_5)$ do not appear in their works.  
Labb\'e, Lapointe and Steiner \cite{LLS} 
stated that the proper subgroup of $G_q(\zeta_n)$ generated by two elements related to the Markoff injectivity conjecture is finite if and only if $n=2,3,4,5$ (\cite[Remark~2.4]{LLS}). The groups $\SL(2, \FF_3)$ and $\SL(2,\FF_5)$ also appear in their work. In addition, we believe that Theorem~\ref{finite} can also be obtained from \cite{FK}, although making this implication explicit seems to require additional effort. Note that \cite{FK} studies the images of Burau's representations of the braid group $B_3$ at $\zeta_n$ for $n=1,2,3,4,5,6$. 
\end{rem}



\section{The 2,3,4,5 property}
In the rest of the paper, for the $n$ under consideration and $z, z' \in \CC$, we write $z \equiv z'$ if $z = \pm \zeta_n^{\, k} z'$ for some $k \in \mathbb{Z}$. For $f(q) \in \ZZ[q]$, we have 
\begin{equation}\label{dual and conjugate}
 f(\zeta_n)^\vee \equiv f(\zeta_n^{\, -1}) = f(\overline{\zeta_n}) = \overline{f(\zeta_n)}.   
\end{equation}

Morier-Genoud and Ovsienko (\cite[Proposition~1.8]{MO1})  pointed out that $\cS_{\frac{r}s}(-1)=0, \pm 1$, and $s (=\cS_{\frac{r}s}(1))$ is even if and only if $\cS_{\frac{r}s}(-1)=0$. 
In \cite[Corollary~7.2]{KMRWY}, it is shown that 
\begin{equation}\label{omega}
\cS_{\frac{r}s}(\omega) \equiv \begin{cases}
0 & \text{if $s \equiv 0 \pmod{3}$,} \\
1  & \text{if $s \not \equiv 0 \pmod{3}$.}
\end{cases}
\end{equation}
Similarly, \cite[Theorem~7.6]{KMRWY} states that 
\begin{equation}\label{zeta_4}
\cS_{\frac{r}s}(i)=0 \Longleftrightarrow \text{$[4]_q \, (=1+q+q^2+q^3)$ divides $\cS_{\frac{r}s}(q)$}\Longleftrightarrow s \in 4\ZZ,  \end{equation}
and \cite[Corollary~7.7]{KMRWY} states that 
\begin{equation}\label{zeta_4 2}
\cS_{\frac{r}s}(i) \equiv \begin{cases}
0 &\text{if $s \equiv 0 \pmod{4}$},\\
1+i  &  \text{if $s \equiv 2 \pmod{4}$},\\
1  &  \text{if $s \equiv 1 \pmod{2}$}. 
\end{cases} 
\end{equation}
Hence, for $n=2,3,4$, we have that $s$ is a multiple of $n$ if and only if $\cS_{\frac{r}s}(\zeta_n)=0$. 
The corresponding statement also holds for $\cR_{\frac{r}s}(q)$.

The following result follows from the computational data obtained in the proof of Theorem~\ref{finite}.

\begin{cor}\label{zeta5}
For an irreducible fraction $\frac{r}s$, we have 
$$
\cS_{\frac{r}s}(\zeta_5) \equiv
\begin{cases}
0 & \text{if $s \in 5\ZZ$ and $r \equiv \pm 1 \pmod{5}$}\\
1-\zeta_5^2 & \text{if $s \in 5\ZZ$ and $r \equiv \pm 2 \pmod{5}$ }\\
1 & \text{if $s \equiv \pm 1 \pmod{5}$}\\
1+\zeta_5 & \text{if $s \equiv \pm 2 \pmod{5}$}
\end{cases}
$$
and 
$$
\cR_{\frac{r}s}(\zeta_5) \equiv
\begin{cases}
0 & \text{if $r \in 5\ZZ$ and $s \equiv \pm 1 \pmod{5}$}\\
1-\zeta_5^2 & \text{if $r \in 5\ZZ$ and $s \equiv \pm 2 \pmod{5}$ }\\
1 & \text{if $r \equiv \pm 1 \pmod{5}$}\\
1+\zeta_5 & \text{if $r \equiv \pm 2 \pmod{5}$.}
\end{cases}
$$ 
Hence $\cS_{\frac{r}s}(\zeta_5)=0$ if and only if $s \in 5\ZZ$ and $r \equiv \pm 1 \pmod{5}$. The same is true for $\cR_{\frac{r}s}(q)$. 
\end{cor}

\begin{rem}\label{remark 5}
(1) The explicit forms of $\cS_{\frac{r}s}(\zeta_5)$ and $\cR_{\frac{r}s}(\zeta_5)$ are considerably more diverse. For example,
$$\zeta_5^{\, 3}(1+\zeta_5)=\zeta_5^{\, 3}+\zeta_5^{\, 4}=\zeta_5^{\, 3}-(1+\zeta_5+\zeta_5^{\, 2}+\zeta_5^{\, 3})=
-(1+\zeta_5+\zeta_5^{\, 2}),$$
$$\zeta_5^{\, 2}(1-\zeta_5^{\, 2})=\zeta_5^{\, 2}-\zeta_5^{\, 4}=\zeta_5^{\, 2}+1+\zeta_5+\zeta_5^{\, 2}+\zeta_5^{\, 3}=
1+\zeta_5+2\zeta_5^{\, 2}+\zeta_5^{\, 3},$$
$\zeta_5^{\, 3}(1-\zeta_5^{\, 2})=\zeta_5^{\, 3}-1$ and 
$\zeta_5^{\, 4}(1-\zeta_5^{\, 2})=-(1+2\zeta_5+\zeta_5^{\, 2}+\zeta_5^{\, 3}).$ 

(2) If we take $\displaystyle \zeta_5=\cos \frac{2\pi}5+i \sin \frac{2\pi}5$, then $\displaystyle 1+\zeta_5 \equiv -\zeta_5^{\, 2}(1+\zeta_5) = \frac{\sqrt{5}+1}2$, where the last quantity is the golden ratio. 

(3) For $0 \ne f(q) \in \ZZ[q]$, let 
\begin{equation}\label{norm}
N(f(\zeta_5)):=\prod_{j=1}^4f(\zeta_5^{\, j}) \in \NN
\end{equation}
be the field norm of $\QQ(\zeta_5)$ over $\QQ$. We can check that $N(1-\zeta_5)=5$ and $N(1+\zeta_5^{\, 2})=1$. By Corollary~\ref{zeta5}, we have
$N(\cS_{\frac{r}s}(\zeta_5)) \le 5$ for all $\frac{r}s$. 
\end{rem}

\begin{cor}\label{PSL(zeta_n)}
Assume that   
$$M_q=\begin{pmatrix}
\cR(q) & \mathcal{V}(q) \\
\cS(q) &  \mathcal{U}(q) \\
\end{pmatrix} \in G_q.$$ 
For $n=2, 3,4,5$, if the natural image of $M_q|_{q=1}$ in $\GL(2,\ZZ/n\ZZ)$ is $\pm E_2$, then $\cR(\zeta_n)=\cU(\zeta_n) =\pm \zeta_n^{\, j}$ for some $0 \le j < n$ and $\cV(\zeta_n)=\cS(\zeta_n)=0$. 
\end{cor}

\begin{proof}
Follows from the computational data obtained in the proof of Theorem~\ref{finite}.
\end{proof}

\begin{rem}\label{PSL(Z/nZ)}
(1)  By Corollary~\ref{PSL(zeta_n)}, we see that $$G_q(\zeta_n)/\<\pm \zeta_n E_2\> \, (=\PSL_q(2,\ZZ)|_{q=\zeta_n})
\cong \PSL(2,\ZZ/n\ZZ)$$ for $n=2,3,4,5$. Moreover, $\PSL(2,\ZZ/3\ZZ) \cong A_4$ and $\PSL(2,\ZZ/5\ZZ) \cong A_5$ are quotients of $G_q(\omega) \cap \SL(2,\CC) \cong \SL(2,\FF_3)$ and $G_q(\zeta_5) \cap \SL(2,\CC) \cong \SL(2,\FF_5)$ respectively. 
On the contrary, $\PSL(2,\ZZ/2\ZZ) \cong S_3$ and 
$\PSL(2,\ZZ/4\ZZ) \cong S_4$ are  {\it not} quotients of $G_q(-1) \cap \SL(2,\ZZ) \cong C_6$ and 
$G_q(i) \cap \SL(2,\CC) \cong \SL(2,\FF_3)$ respectively, just quotients of $G_q(-1)$ and $G_q(i)$.    

(2) For a general $n$, the statement of Corollary~\ref{PSL(zeta_n)} does not hold. 
 As ideals of a polynomial ring $\ZZ[q]$, we have $(q-1, [n]_q)=(q-1,n)$. If $n$ is prime, then $[n]_q$ coincides with the $n$-th cyclotomic polynomial, and we have the surjective ring homomorphism $\varphi: \ZZ[\zeta_n] \, ( \cong \ZZ[q]/([n]_q)) \too \ZZ/n\ZZ$ given by $\overline{q} \longmapsto \overline{1}$, which induces the group homomorphism $\Phi: G_q(\zeta_n) \too \PSL(2, \ZZ/n\ZZ)$. By Proposition~\ref{PSLq}, $\Phi$ is surjective. 
Since $|G_q(\zeta_n)|=\infty$ for $n \ge 6$ by Theorem~\ref{finite} and $|\PSL(2,\ZZ/n\ZZ)|<\infty$, 
we have that $\Ker \Phi$ is an infinite group  for a prime integer $n \ge 7$. This means that the statement of Corollary~\ref{PSL(zeta_n)} does not hold also for a prime $n \ge 7$. This statement does not hold also for $n=6$. For example, $X_q:=M_q(2,4,3,2) \cdot R_q$ satisfies 
$X_q|_{q=1}=\begin{pmatrix}31&12\\ 18 & 7\end{pmatrix}$ and its image  $\GL(2,\ZZ/6\ZZ)$ is $E_2$. However, $X_q|_{q=\zeta_6} \not \equiv \pm E_2$.  
\end{rem}

For $f(q) \in \ZZ[q]$,  
we say that $f(q)$ has the {\it 2,3,4,5 property} if 
$$ \text{$f(1)$ is a multiple of $n$} \Longleftrightarrow f(\zeta_n)=0$$
holds for $n=2,3,4,5$ (possibly under a mild assumption or with a few exceptions). 
When $n$ is a prime integer, we have $f(\zeta_n)=0$  if and only if $[n]_q \mid f(q)$. Hence, for $n=2,3,5$, the implication $\Leftarrow$ obviously holds. The above results state that $\cS_{\frac{r}s}(q)$ satisfies the 2,3,4,5 property, although an additional assumption is required for $n=5$. 

Recall that $J_{\frac{r}s}(q) \in \ZZ[q]$ is the normalized Jones polynomial of the rational link $L(\frac{r}s)$. 
We have  $J_{\frac{r}s}(-1) \ne 0$ and $J_{\frac{r}s}(\omega)\ne 0$ for all $\frac{r}s$, and $J_{\frac{r}s}(i) =0$ if and only if $r \equiv 2 \pmod{4}$ (recall that $J_{\frac{r}s}(1)=r$). See, for example, \cite[\S8]{KMRWY}. 
So we cannot say $J_{\frac{r}s}(q)$ satisfies the 2,3,4,5 property. However, the following holds.   

\begin{cor}\label{J zeta5}
We have     
$$
J_{\frac{r}s}(\zeta_5) \equiv
\begin{cases}
0 & \text{if $r \in 5\ZZ$ and $s \equiv \pm 2 \pmod{5}$}\\
1-\zeta_5 & \text{if $r \in 5\ZZ$ and $s \equiv \pm 1 \pmod{5}$ }\\
1 & \text{if $r \equiv \pm 1 \pmod{5}$}\\
1+\zeta_5^{\, 2} \, (=(1+\zeta_5)^{-1})& \text{if $r \equiv \pm 2 \pmod{5}$.}
\end{cases}
$$ 
Hence $J_{\frac{r}s}(\zeta_5)=0$ if and only if $r \in 5\ZZ$ and $s \equiv \pm 2 \pmod{5}$. 
\end{cor}

\begin{proof}
By the first equality of \eqref{Jones and flat}, the assertion follows from the computational data in the proof of Theorem~\ref{finite}. The relevant files are available at the above address. 
\end{proof}

By Theorem~\ref{J-L}, we have  
$\flS_{\frac{r}s}(-1) \ne 0$ and $\flS_{\frac{r}s}(\omega)\ne 0$ for all $\frac{r}s$, $\flS_{\frac{r}s}(i) =0$ if and only if $s \equiv 2 \pmod{4}$, and  $\flS_{\frac{r}s}(\zeta_5)=0$ if and only if $s \in 5\ZZ$ and $r \equiv \pm 2 \pmod{5}$.

For $M_q \in G_q$, set $f(q):=\Tr M_q$. Then we have $f(-1)=0, \pm 1, \pm 2$. Hence, $f(-1)$ can be non-zero even if $f(1)$ is even. However, $f(q)=\Tr M_q$ still satisfies a weaker form of the 2,3,4,5 property. 

\begin{cor}\label{trace for 3,4,5}
The following hold. 
\begin{itemize}
\item[(1)] For $\omega=\zeta_3$, we have  $\{ \, \operatorname{Tr} M \mid M \in G_q(\omega) \, \}= \{0, \pm \omega^j, \pm 2 \omega^j \mid j=0,1,2 \}.$ 
Similarly,  
\begin{align*}
&\{ \, \operatorname{Tr} M \mid M \in G_q(i) \, \}\\
&=\{0, \pm 1, \pm 2, \pm i, \pm 2i, \pm1 \pm i \text{ (with arbitrary signs) }\}\\
&=\{0, c, \sqrt{2} \, \zeta_8\, c, 2c \mid c=i^j, j= 0,1,2,3 \} 
\end{align*}
and 
\begin{align*}
&\{ \, \operatorname{Tr} M \mid M \in G_q(\zeta_5) \, \} \\
&=\left\{0, c, 2c, (1+\zeta_5)c,  (1+\zeta_5^2)c\,\middle|\, c=\pm\zeta_5^j, j=0,1,2,3,4\right\}\\
&=\left\{0, \frac{\sqrt{5}-1}2c, c, \frac{\sqrt{5}+1}2c, 2c\,\middle|\, c=\pm\zeta_5^j, j=0,1,2,3,4\right\}. 
\end{align*}
\item[(2)] For $M_q \in G_q$, set $f(q):=\Tr M_q$. 
For $n=3,4,5$, if $f(1) \in n\ZZ$ then $f(\zeta_n)=0$. For $n=3,5$, the converse also holds.  
\end{itemize}
\end{cor}

\begin{proof}
(1) Immediate from the computation given in the proof of Theorem~\ref{finite}, 

(2) The assertion directly follows from (1) for $n=3,5$,  and it remains to show the case $n=4$.  For contradiction, assume that $f(1)$ is a multiple of 4, but $f(i) \ne 0$. Set
$$M_q=\begin{pmatrix}
\cR(q) & \cV(q) \\
\cS(q)  & \cU(q) \\
\end{pmatrix}.$$
By the present assumption and \cite[Corollary~7.7]{KMRWY}, $M_q|_{q=i}$ is one of the following matrices 
$$i^j \begin{pmatrix}
1& 0 \\ 0&1 \\
\end{pmatrix}, \ 
i^j  \begin{pmatrix}i &1\\ 0 & 1\end{pmatrix}, \ i^j \begin{pmatrix}
1& -1 \\ 0& i \\
\end{pmatrix}, \ i^j \begin{pmatrix}1 &0\\ -i & i\end{pmatrix}, \ i^j \begin{pmatrix}
i& 0 \\ i & 1 \\
\end{pmatrix} \quad (j=0,1,2,3).$$
In any of these cases, at least one of  $\cS(i)$ and $\cV(i)$ is 0, 
that is, at least one of  $\cS(1)$ and $\cV(1)$ is a multiple of 4. Since $M_q|_{q=1} \in \mathrm{SL}(2,\ZZ)$, we have $\cR(1) \cdot \cU(1) \equiv 1 \pmod{4}$, and it implies that  $\cR(1) \equiv \cU(1) \equiv \pm 1 \pmod{4}$. Hence 
$$\Tr(M_q|_{q=1}) = \cR(1)+\cU(1) \equiv 2 \pmod{4},$$
and this is a contradiction. 
\end{proof}

\begin{rem}
(1) With the notation of Corollary~\ref{trace for 3,4,5} (2), $f(i)=0$ does not necessarily imply $f(1) \in 4\ZZ$, although $f(1) \in 2\ZZ$ in this case.  For example, $\Tr(M_q(2,2))=q^2+1$.

(2)  With the above notation, \cite[Theorem~12]{EJMV} (see also \cite{K}) shows that  if $M_q$ is a $q$-{\it Cohn matrix}, then $f(\omega)=0$. The $q$-Cohn matrices constitute a further special subclass of the matrices $M_q \in G_q$ with $f(1) \in 3\ZZ$. If $M_q$ is $q$-Cohn, then $\Tr M_q/[3]_q$ has a special meaning. It is not clear how much / in what way this result can be generalized.
\end{rem}

For an irreducible fraction $\frac{r}s >1$, the triplet $(a,b,c)$ given by  
\begin{equation}\label{abc}
a:=2rs, \quad b:=r^2-s^2, \quad c:=r^2+s^2
\end{equation}
is a Pythagorean triple (i.e., $a^2+b^2=c^2$) with $\operatorname{gcd}(a,b,c)=1,2$. Mathevet, Morier-Genoud and Ovsienko \cite{MMO} quantize \eqref{abc}. More precisely, they set
$$\cA_{\frac{r}s}(q):=q \cR_{\frac{r}s}(q) \cR_{\frac{s}r}(q)+ \cS_{\frac{r}s}(q) \cS_{\frac{s}r}(q), \quad \cB_{\frac{r}s}(q):= \cR_{\frac{r}s}(q) \cS_{\frac{s}r}(q)- \cS_{\frac{r}s}(q) \cR_{\frac{s}r}(q),$$
$$\cC_{\frac{r}s}(q)=q \cdot \cR_{\frac{r}s}(q)^2+ \cS_{\frac{r}s}(q)^2.$$
They are monic polynomials whose constant terms are 1. Moreover, they satisfy $\cA_{\frac{r}s}(1)=2rs$, $\cB_{\frac{r}s}(1)=r^2-s^2$, $\cC_{\frac{r}s}(1)=r^2+s^2$, and 
$$\cA_{\frac{r}s}(q)^2 + q \cB_{\frac{r}s}(q)^2 = \cC_{\frac{r}s}(q) \cdot 
\cC_{\frac{r}s}(q)^\vee.$$

The following result states that $\cA_{\frac{r}s}(q)$ and $\cB_{\frac{r}s}(q)$ satisfy the 2,3,4,5 property without any assumptions or exceptions.

\begin{thm}\label{ABC}
For $n=2,3,4,5$, $\cA_{\frac{r}s}(1)$ (resp. $\cB_{\frac{r}s}(1)$) is a multiple of $n$ if and only if $\cA_{\frac{r}s}(\zeta_n)=0$ (resp. $\cB_{\frac{r}s}(\zeta_n)=0$).     
\end{thm}

\begin{proof}
Recall that the ``if'' part is clear for $n=2,3,5$. 
By \eqref{s/r} and \eqref{dual and conjugate} we have 
$$\cA_{\frac{r}s}(\zeta_n) \equiv \zeta_n \cR_{\frac{r}s}(\zeta_n) 
\overline{\cS_{\frac{r}s}(\zeta_n)}+ \cS_{\frac{r}s}(\zeta_n) \overline{\cR_{\frac{r}s}(\zeta_n)},$$  
\begin{equation}\label{cB(q)}
\cB_{\frac{r}s}(\zeta_n)\equiv\cR_{\frac{r}s}(\zeta_n) \overline{\cR_{\frac{r}s}(\zeta_n)}- \cS_{\frac{r}s}(q) \overline{\cS_{\frac{r}s}(\zeta_n)}=
|\cR_{\frac{r}s}(\zeta_n)|^2 - |\cS_{\frac{r}s}(\zeta_n)|^2
\end{equation}

\medskip

\noindent {\bf The case $n=2$:} Since 
$$\begin{pmatrix}\cR_{\frac{r}s}(-1) \\ \cS_{\frac{r}s}(-1)\end{pmatrix} \equiv \begin{pmatrix}1 \\0\end{pmatrix}, \begin{pmatrix}0 \\1\end{pmatrix}, \begin{pmatrix}1 \\ 1\end{pmatrix},$$
the assertion follows. 
\medskip

\noindent {\bf The case $n=3$:} 
The assertion for $\cA_{\frac{r}s}(q)$ is clear from \cite[Corollary~7.2]{KMRWY}. For $\cB_{\frac{r}s}(q)$, it suffices to show that if $\cB_{\frac{r}s}(1)$ is a multiple of 3 then  $\cB_{\frac{r}s}(\omega)=0$.  Since $\frac{r}s$ is irreducible, $\cB_{\frac{r}s}(1)$ is a multiple of 3 if and only if $r,s \not \in 3\ZZ$. In this case, we have 
$$\begin{pmatrix}\cR_{\frac{r}s}(\omega) \\ \cS_{\frac{r}s}(\omega)\end{pmatrix} \equiv \begin{pmatrix}1 \\ -\omega\end{pmatrix}, \begin{pmatrix}1 \\ 1\end{pmatrix}$$
by \cite[Theorem~7.1]{KMRWY}, 
and we have $\cB_{\frac{r}s}(\omega)=0$ in both cases by \eqref{cB(q)}.

\medskip

\noindent {\bf The case $n=4$:} 
We have 
$$\begin{pmatrix}\cR_{\frac{r}s}(i) \\ \cS_{\frac{r}s}(i)\end{pmatrix}\equiv 
\begin{pmatrix}1\\0\end{pmatrix},
\begin{pmatrix}0\\1\end{pmatrix}, 
\begin{pmatrix}1\\1\end{pmatrix},
\begin{pmatrix}i\\1\end{pmatrix}, 
\begin{pmatrix}1+i\\1\end{pmatrix}, 
\begin{pmatrix}1\\1-i\end{pmatrix},$$ 
and $\cA_{\frac{r}s}(i) \equiv i \cR_{\frac{r}s}(i)\overline{\cS_{\frac{r}s}(i)}+ \cS_{\frac{r}s}(i) \overline{\cR_{\frac{r}s}(i)}$  is equivalent to 0,0, $i+1$, $-1-i$, 0, 0 respectively.  Hence $\cA_{\frac{r}s}(1)=2rs \in 4\ZZ$, if and only if one of $r$ and $s$ is even, if and only if $\cA_{\frac{r}s}(i)=0$. The assertion follows for $\cB_{\frac{r}s}(i)$ follows from \eqref{cB(q)}. 

\medskip

\noindent {\bf The case $n=5$:} 
 The assertion for  $\cA_{\frac{r}s}(q)$ can be verified by computer. We can also prove the assertion for $\cB_{\frac{r}s}(q)$ in this way, but we have a more conceptual proof. 
Since $\frac{r}s$ is irreducible, we have $r^2 \equiv s^2 \pmod{5}$, if and only if $r \equiv \pm s \not \equiv 0 \pmod{5}$, if and only if $|\cR_{\frac{r}s}(\zeta_5)|= |\cS_{\frac{r}s}(\zeta_5)|$, if and only if $\cB_{\frac{r}s}(\zeta_5)=0$, by Corollary~\ref{zeta5}. 
\end{proof}

\begin{rem}\label{ABC lem}
(1) The above proof shows that
$$\cA_{\frac{r}s}(-1)\equiv 0, \quad \cB_{\frac{r}s}(-1) \equiv 0,1, \quad \cA_{\frac{r}s}(\omega)\equiv 0,1+\omega \,  (=\sqrt{3}\, \zeta_{12}), \quad \cB_{\frac{r}s}(\omega) \equiv 0,1,$$
$$\cA_{\frac{r}s}(i)\equiv 0,1+i \,  (=\sqrt{2}\, \zeta_{8}), \quad \cB_{\frac{r}s}(i) \equiv 0,1,  \quad 
\cA_{\frac{r}s}(\zeta_5)\equiv 0,1, 1+q,  \quad \cB_{\frac{r}s}(\zeta_5) \equiv 0,1, 1+q.$$

(2) It is easy to check that $\cC_{\frac{r}s}(-1)=0$, if and only if $\cC_{\frac{r}s}(1)$ is even. However, we always have $\cC_{\frac{r}s}(\zeta_n)\ne 0$ for $n=3,4,5$. 
The cases $n=3, 4$ are easy. 
If $\cC_{\frac{r}s}(\zeta_5)=0$, then we have $\cC_{\frac{r}s}(1) =r^2+s^2\in 5\ZZ$ and $(r,s)\equiv (\pm 1, \pm 2),  (\pm 2, \pm 1) \pmod{5}$. Hence 
$$\cC_{\frac{r}s}(\zeta_5)=\zeta_5 \cdot \cR_{\frac{r}s}(\zeta_5)^2+ \cS_{\frac{r}s}(\zeta_5)^2 \ne 0.$$
In fact, by Remark~\ref{remark 5}, one of $N(\zeta_5 \cdot \cR_{\frac{r}s}(\zeta_5)^2)$ and $N(\cS_{\frac{r}s}(\zeta_5)^2)$ is equal to 1, and the other is 25. Here $N(-)$ is the norm of $\QQ(\zeta_5)/\QQ$ given in \eqref{norm}. 
\end{rem}


\section{The case \texorpdfstring{$n=6$}{n=6} and a conjecture}
\begin{thm}\label{TR-omega}
Note that $\zeta_6 =-\omega$.  We have 
$$
\{ \, \operatorname{Tr} M \mid M \in G_q(-\omega) \, \}
= \{ \, 0,  c, \sqrt{3}  \, \zeta_{12}\, c, 2c \mid c=(-\omega)^j \, \text{ for } \,  j =0,1,\ldots,5 \, \}.$$
 In particular, we have $|\{ \, \operatorname{Tr} M \mid M \in G_q(-\omega) \, \}| < \infty$. 
\end{thm}

\begin{proof}
The following two elements
$$R_{-\omega}:=R_q|_{q=-\omega}=\left( \begin{array}{cccc}
-\omega&1\\0&1 \end{array}\right) \quad \text{and} \quad 
S_{-\omega}:=S_q|_{q=-\omega}=\left( \begin{array}{cccc}0&\omega^2\\1&0 \end{array}\right) 
$$
generate $G_q(-\omega)$. 
For $P=\begin{pmatrix}1 &1 \\ -\omega^2&0 \end{pmatrix}$, we have 
$$P^{-1}R_{-\omega}P=\left( \begin{array}{cccc}
1&0\\0&-\omega \end{array}\right) \quad \text{and} \quad 
P^{-1}S_{-\omega}P
=\left( \begin{array}{cccc}-\omega&-\omega\\0&\omega \end{array}\right). 
$$
Hence the group $G_q(-\omega)$ is simultaneously upper-triangularizable by $P$, and any element of $P^{-1}G_q(-\omega)P$ is of the forms 
$$(-\omega)^j\left( \begin{array}{cccc}
(-\omega)^k&*\\0&1 \end{array}\right)$$
for $1 \le j,k \le 5$, and any choice of $j,k$ actually appears in $G_q(-\omega)$.  
Since $1+\omega=-\omega^2$, 
$1-\omega=\sqrt{3}\cdot \zeta_{12}$, 
$1+\omega^2=-\omega$ and $1-\omega^2=-\omega^2(1-\omega)$, the assertion follows. 
\end{proof}

\begin{rem}
Since $G_q(-\omega)$ is simultaneously upper-triangularizable, it is solvable. This recovers the corresponding statement in \cite{TU25}. Note that $G_q(\zeta_n)$ is solvable for $n=2,3,4$, but $G_q(\zeta_5)$ is not.
\end{rem}

\begin{cor}\label{trace finite}
For $\zeta \in \CC^*$,  
the set $\{ \Tr M \mid M \in G_q(\zeta)\}$ is finite if and only if $\zeta=\zeta_n$ for $n=2,3,4, 5, 6$.    
\end{cor}

\begin{proof}
Sufficiency: If $n =2,3,4,5$, then $G_q(\zeta_n)$ itself is finite by Theorem~\ref{finite}, 
and the case $n=6$ follows from Theorem~\ref{TR-omega}. 

Necessity: If $\zeta \ne \zeta_n$ for all $n \ge 2$, then $\{ \Tr M \mid M \in G_q(\zeta)\}$ is clearly infinite (consider $R_q^{\, n}|_{q=\zeta} \in G_q(\zeta)$). The case $\zeta=\zeta_n$ for $n \ge 7$ follows from the final step of the proof of Theorem~\ref{finite}.  
\end{proof}

\begin{prop}\label{multiple of 6}
If $\cS_{\frac{r}s}(-\omega)=0$, then $[6]_q=1+q+\cdots +q^5$ divides $\cS_{\frac{r}s}(q)$, and hence $s$ is a multiple of $6.$  
\end{prop}

\begin{proof}
Assume that $\cS_{\frac{r}s}(-\omega)=0$. Then we have $\cS_{\frac{r}s}(q)=(q^2-q+1)\cdot f(q)$ for some $f(q) \in \ZZ[q]$. 
Since  $\cS_{\frac{r}s}(-1)=\{ 0, \pm 1 \}$  by \cite[Proposition 1.8]{MO1}, $f(-1) \in \ZZ$ and $(-1)^2-(-1)+1=3$, we have $f(-1)=0$, and $q+1$ divides $f(q)$. 
By \cite[Theorem~7.1]{KMRWY} (i.e., \eqref{omega} of this paper), we have $|\cS_{\frac{r}s}(\omega)| \le 1$.  Clearly,  $f(\omega) \in \ZZ[\omega]$, $|f(\omega)|^2 \in \ZZ$, and $\omega^2-\omega +1=-2 \omega$, so it follows that $f(\omega)=0$, and hence $q^2+q+1$ divides $f(q)$. Putting the above together, $(q+1)(q^2+q+1)(q^2-q+1)=[6]_q$ divides $\cS_{\frac{r}s}(q)$. 
\end{proof}

\begin{rem}
(1) The condition $s \in 6\ZZ$ does not necessarily imply $\cS_{\frac{r}s}(-\omega)=0$. For example, we have $\cS_{\frac5{12}}(q)
=(q+1)(q^2+1)(q^2+q+1)$. 
Of course, $\cS_{\frac{r}{s}}(-\omega)=0$ might happen. A trivial   example is $\cS_{\frac{1}{6n}}(q)=[6n]_q=1+q+\cdots+q^{6n-1}$  for a positive integer $n$. Another example is $\cS_{\frac{5}{24}}(q)=
(q+1)^2(q^2+1)(q^2-q+1)(q^2+q+1)$. 

(2) For $M_q \in G_q$, set $f(q):= \Tr M_q$. If $f(-\omega)=0$, then $f(1)$ is even. In fact, since $|f(-1)|\le 2$ and $(-1)^2-(-1)+1=3$, we can use the argument in the first half of the proof of Proposition~\ref{multiple of 6} here.  However, since the case $f(\omega)=2 \omega^j$ is possible, the second half of the proof does not go through.
Therefore, $f(1)$ need not be a multiple of 3, even if $f(-\omega)=0$. For example, we have $\Tr M_q(2,2,2)=(q+1)(q^2-q+1)$ (clearly $f(1)=2$).
\end{rem}

\begin{cor}\label{flat finite}
For $\zeta \in \CC^*$, the following are equivalent. 
\begin{itemize} 
\item[(1)]  The set $\{ \, \flS_{\frac{r}s}(\zeta) \mid \frac{r}s >1 \, \}$ is finite. 
\item[(2)]  The set $\{ \, J_{\frac{r}s}(\zeta) \mid \frac{r}s >1 \, \}$ is finite. 
\item[(3)] $\zeta =\zeta_n$ for $n=2,3,4,5,6$. 
\end{itemize}
\end{cor}

\begin{proof}
The equivalence between (1) and (2) is immediate from Theorem~\ref{J-L} and \eqref{R and S for flat}. So we will prove the equivalence between (1) and (3) freely using  (1) $\Longleftrightarrow$ (2). 

Since $J_{\frac{n}1}(q)=1+q^2+q^3+\cdots+q^n$ for $n \ge 2$, $\{ J_{\frac{r}s}(\zeta) \mid \frac{r}s > 1\}$ is infinite (equivalently, $\{\flS_{\frac{r}s}(\zeta) \mid \frac{r}s > 1\}$ is infinite) if  $\zeta \ne \zeta_n$ for all $n \ge 2$. So we may assume that $\zeta =\zeta_n$ for some $n \ge 2$. 
By \eqref{R and S for flat},  $\{ \flS_{\frac{r}s}(\zeta_n) \mid  \frac{r}s >1 \}$ is finite,  if and only if so is $\{ \flR_{\frac{r}s}(\zeta_n) \mid  \frac{r}s >1 \}$, if and only if
so is  $\{ \zeta_n^{\, v(r/s)} \flS_{\frac{r}s}(\zeta_n)^\vee \mid  \frac{r}s >1 \}$ in the notation of \eqref{Thomas}. 

It is a classical result that $\{ \, J_{\frac{r}s}(-\omega) \mid \frac{r}s >1 \, \}$ is finite (c.f. the last part of \cite{KMRWY}). So we may assume that $n \ne 6$.   By \eqref{Thomas}, for $\alpha=\frac{r}s$, we have 
\begin{equation}\label{Thomas dual}
\begin{pmatrix}
\Delta(q) \cdot \cR_\alpha(q) \\
\Delta(q) \cdot \cS_\alpha(q)
\end{pmatrix}=
\begin{pmatrix}
1 & q-1 \\
1-q &q
\end{pmatrix}
\begin{pmatrix}
\flR_\alpha(q)^\vee \\
q^{v(\alpha)}\flS_\alpha(q)^\vee
\end{pmatrix},
\end{equation}
where $\Delta(q):= q^2-q+1$. Since $n \ne 6$ now,  we have $\Delta(\zeta_n)\ne 0$.  By the above observations, 
$\{ \, \flS_{\frac{r}s}(\zeta_n) \mid \frac{r}s >1 \, \}$ is finite if and only if so is  $\{ \, \cS_{\frac{r}s}(\zeta_n) \mid \frac{r}s >1 \, \}$. Hence, the assertion follows from Corollary~\ref{S finite}.
\end{proof}

Let $\Phi_n(q)$ be the $n$-th cyclotomic polynomial. 
We now restate Conjecture~\ref{[n]_q divides}. 

\medskip

\noindent{\bf Conjecture~\ref{[n]_q divides}.} 
For $n \ge 2$, if $\Phi_n(q)$ divides 
$\cS_{\frac{r}s}(q)$, then $[n]_q$ divides $\cS_{\frac{r}s}(q)$.  In particular, $\cS_{\frac{r}s}(\zeta_n)=0$ implies $s \in n \ZZ$.

\medskip

The following is an evidence of the conjecture.

\begin{prop}\label{evidence}
Conjecture~\ref{[n]_q divides} holds if  $n$ is a  prime number or $n=4,6,8,9,10,15, 25$.  
\end{prop}

\begin{proof}
If $n$ is prime, the assertion is obvious. The case $n=4$ was proved in  \cite[Theorem~7.6]{KMRWY} (i.e., the equation \eqref{zeta_4} of this paper).  The case $n=6$ was established in Proposition~\ref{multiple of 6}. Let us prove the remaining cases. 
Assume that $\cS_{\frac{r}s}(q)=\Phi_n(q)\cdot f(q)$ for some $f(q) \in \ZZ[q]$. In the remainder of this proof, we use this notation. 

\medskip

\noindent {\bf The case $n=8$:} Note that $\Phi_8(q)=q^4+1$. 
Since  $|\cS_{\frac{r}s}(i)| \le \sqrt{2}$  by \cite[Corollary~7.7]{KMRWY} (i.e., the equation \eqref{zeta_4 2} of this paper), $|f(i)|^2 \in \ZZ$, and $\Phi_8(i)=2$, we have $f(i)=0$, and hence $\cS_{\frac{r}s}(i)=0$.   
It implies that $[4]_q$ divides $\cS_{\frac{r}s}(q)$ by \eqref{zeta_4}. 
Putting the above together, $[8]_q=\Phi_4(q) \cdot [4]_q$ also divides $\cS_{\frac{r}s}(q)$.

\medskip

\noindent {\bf The case $n=9$:} Note that $\Phi_9(q)=q^6+q^3+1$. 
Since  $|\cS_{\frac{r}s}(\omega)| \le 1$  by \eqref{omega}, 
$|f(\omega)|^2 \in \ZZ$, and $\Phi_9(\omega)=3$, we have $f(\omega)=0$, that is, $[3]_q$ divides $f(q)$.  Hence, $[9]_q=\Phi_9(q) \cdot [3]_q$ divides $\cS_{\frac{r}s}(q)$.

\medskip

\noindent {\bf The case $n=10$:} Note that $\Phi_{10}(q)=q^4-q^3+q^2-q+1$. Since $\Phi_{10}(-1)=5$, $[2]_q=q+1$ divides $f(q)$ by the above argument. 
We can check that $N(\Phi_{10}(\zeta_5))=16$, where $N(-)$ is the norm of $\QQ(\zeta_5)/\QQ$ given in \eqref{norm}. By Remark~\ref{remark 5} (3), we have $f(\zeta_5)=0$, and $[5]_q$ divides $f(q)$.  Since $[10]_q=\Phi_{10}(q)\cdot[2]_q \cdot [5]_q$, we are done.

\medskip

\noindent {\bf The case $n=15$:} Note that $\Phi_{15}(q)=q^8-q^7+q^5-q^4+q^3-q+1$. Since $|\Phi_{15}(\omega)|=|5\omega|=5$, we have that $[3]_q$ divides $f(q)$ by the above argument. Similarly, since  $N(\Phi_{15}(\zeta_5))=81$, $[5]_q$ divides $f(q)$.  Since $[15]_q=\Phi_{15}(q)\cdot[3]_q \cdot [5]_q$, we are done.

\medskip

\noindent {\bf The case $n=25$:} Note that $\Phi_{25}(q)=q^{20}+q^{15}+q^{10}+q^5+1$. Since we can check that  $N(\Phi_{25}(\zeta_5))=625$, $[5]_q$ divides $f(q)$. Since $[25]_q=\Phi_{25}(q)\cdot [5]_q$, we are done. 
\end{proof}

\begin{prop}\label{12,20}
For $n=12,20$, $\cS_{\frac{r}s}(\zeta_n)=0$ implies $s \in n \ZZ$.     
\end{prop}

\begin{proof}
We argue as in the proof of Proposition~\ref{evidence}, using the same notation. 

\noindent {\bf The case $n=12$:} Note that $\Phi_{12}(q)=q^4-q^2+1$. Since $\Phi_{12}(i)=3$ (resp. $\Phi_{12}(\omega)=2 \omega+2$), $s$ is a multiple of 4 (resp. 3). Hence $s \in 12\ZZ$. 

\noindent {\bf The case $n=20$:} Note that $\Phi_{20}(q)=q^8-q^6+q^4-q^2+1$. Since $\Phi_{20}(i)=5$ (resp. $N(\Phi_{20}(\zeta_5)=16$), $s$ is a multiple of 4 (resp. 5). Hence $s \in 20\ZZ$. 
\end{proof}

\begin{rem}
(1)  Even if $\cS_{\frac{r}s}(\zeta_{12})=0$ (resp. $\cS_{\frac{r}s}(\zeta_{20})=0$), we cannot prove that $\cS_{\frac{r}s}(\zeta_6)=0$ (resp. $\cS_{\frac{r}s}(\zeta_{10})=0$), so it is still open whether $[12]_q$ (resp. $[20]_q$) divides $\cS_{\frac{r}s}(\zeta_q)$. 

(2) Since $\Phi_{30}(q)=q^8+q^7-q^5-q^4-q^3+q+1$ and  $\Phi_{30}(-1)=\Phi_{30}(i)=1$, it is not even known whether $\cS_{\frac{r}s}(\zeta_{30})=0$ implies $\cS_{\frac{r}s}(-1)=0$. 
On the other hand, since we have no results on $\{ \cS_{\frac{r}s}(\zeta_n) \mid \frac{r}s > 1 \}$ for $n \ge 7$, the argument used in the proofs of the above propositions cannot be applied if $s$ has a prime factor $p\ge7$.
\end{rem}

For the quantized triple $(\cA_{\frac{r}s}(q), \cB_{\frac{r}s}(q), \cC_{\frac{r}s}(q))$ introduced by \cite{MMO}, it is expected that $\cA_{\frac{r}s}(q)$ and $\cB_{\frac{r}s}(q)$ satisfy the property corresponding to Conjecture~\ref{[n]_q divides}. By Remark~\ref{ABC lem} (1), the statements corresponding to Propositions~\ref{evidence} and \ref{12,20} hold for  $\cA_{\frac{r}s}(q)$ and $\cB_{\frac{r}s}(q)$.


\medskip

\section*{Acknowledgments}
We are grateful to Yuta Hatasa and Natsuki Sanada  for valuable comments and assistance in the early stages of this project. We also appreciate the valuable comments or  encouragements by Professors Mitsuyasu Hashimoto, Akihiro Munemasa, Atsushi Takahashi, and Michihisa Wakui. We thank Professors S\'ebastien Labb\'e and Muhammed Uludağ for bringing their work to our attention.
Finally, we are grateful to the anonymous referee of an earlier version of this work for valuable comments, suggestions, and references, which led to a substantial revision of the present paper.

T. B is partially supported by the Kansai University Grantin-Aid for research progress in the graduate course, 2025. X. R is partially supported by JSPS KAKENHI Grant Number 21H04994, 19K03456. K. Y is partially supported by 25K06928. 

\medskip


\end{document}